\documentclass{amsart}
\usepackage{setspace}
\usepackage{a4}
\usepackage{amsthm}
\usepackage{latexsym}
\usepackage{amsfonts}
\usepackage{graphicx}
\usepackage{textcomp}
\usepackage{cite}
\usepackage{enumerate}
\usepackage{amssymb}
\usepackage{hyperref}
\usepackage{amsmath}
\usepackage{tikz}
\usepackage[mathscr]{euscript}
\usepackage{mathtools}
\newtheorem{theorem}{Theorem}[section]

\newtheorem{definition}[theorem]{Definition}
\newtheorem{example}[theorem]{Example}

\newtheorem{question}[theorem]{Question}
\title{This is the title}
\usepackage{fancyhdr}

\begin{document}
\hrule\hrule\hrule\hrule\hrule
\vspace{0.3cm}	
\begin{center}
{\bf{3-Heisenberg-Robertson-Schrodinger  Uncertainty Principle}}\\
\vspace{0.3cm}
\hrule\hrule\hrule\hrule\hrule
\vspace{0.3cm}
\textbf{K. Mahesh Krishna}\\
School of Mathematics and Natural Sciences\\
Chanakya University Global Campus\\
NH-648, Haraluru Village\\
Devanahalli Taluk, 	Bengaluru  Rural District\\
Karnataka  562 110 India\\
Email: kmaheshak@gmail.com\\

Date: \today
\end{center}

\hrule\hrule
\vspace{0.5cm}
\textbf{Abstract}:  Let  $\mathcal{X}$ be a 3-product space. Let  $A:	\mathcal{D}(A)\subseteq \mathcal{X}\to  \mathcal{X}$,  $B:	\mathcal{D}(B)\subseteq \mathcal{X}\to  \mathcal{X}$ and $C:	\mathcal{D}(C)\subseteq \mathcal{X}\to  \mathcal{X}$ be  possibly unbounded 3-self-adjoint operators. Then for all 
\begin{align*}
	x \in \mathcal{D}(ABC)\cap\mathcal{D}(ACB) \cap \mathcal{D}(BAC)\cap\mathcal{D}(BCA) \cap \mathcal{D}(CAB)\cap\mathcal{D}(CBA)
\end{align*}
with $\langle x, x, x \rangle =1$, we show that 
\begin{align*}
(1)\quad \quad 	\Delta _x(3, A)	\Delta _x(3, B)	\Delta _x(3, C)\geq |\langle (ABC-a BC-b AC-c AB)x, x, x\rangle +2abc|,
\end{align*}
where 
\begin{align*}
		\Delta _x(3, A)\coloneqq \|Ax-\langle Ax, x, x \rangle x \|, \quad a\coloneqq \langle Ax, x, x \rangle, \quad b \coloneqq \langle Bx, x, x \rangle, \quad c \coloneqq \langle Cx, x, x \rangle.
\end{align*}
We call Inequality (1) as 3-Heisenberg-Robertson-Schrodinger  uncertainty principle. Classical Heisenberg-Robertson-Schrodinger  uncertainty principle (by Schrodinger in 1930) considers two operators whereas Inequality (1) considers three operators. 

\textbf{Keywords}:   Uncertainty Principle,  Banach space.

\textbf{Mathematics Subject Classification (2020)}: 46C50, 46B99.\\

\hrule

\hrule
\section{Introduction}
Let $\mathcal{H}$ be a complex Hilbert space and $A$ be a possibly unbounded self-adjoint linear operator defined on domain $\mathcal{D}(A)\subseteq \mathcal{H}$. For $h \in \mathcal{D}(A)$ with $\|h\|=1$, define the \textbf{uncertainty} (also known as variance) of $A$ at the point $h$ as 
\begin{align*}
	\Delta _h(A)\coloneqq \|Ah-\langle Ah, h \rangle h \|=\sqrt{\|Ah\|^2-\langle Ah, h \rangle^2}. 
\end{align*}
In 1929, Robertson \cite{ROBERTSON} derived the following mathematical form of the uncertainty principle (term due to Condon \cite{CONDON}) of Heisenberg derived in 1927 \cite{HEISENBERG}. Recall that, for two linear operators $A:	\mathcal{D}(A)\subseteq \mathcal{H}\to  \mathcal{H}$ and $B:	\mathcal{D}(B)\subseteq \mathcal{H}\to  \mathcal{H}$, we define $[A,B] \coloneqq AB-BA$ and $\{A,B\}\coloneqq AB+BA$.
\begin{theorem} \cite{ROBERTSON, CASSIDY, HEISENBERG, VONNEUMANNBOOK, DEBNATHMIKUSINSKI, OZAWA} (\textbf{Heisenberg-Robertson Uncertainty Principle}) \label{HRT}
Let  $A:	\mathcal{D}(A)\subseteq \mathcal{H}\to  \mathcal{H}$ and $B:	\mathcal{D}(B)\subseteq \mathcal{H}\to  \mathcal{H}$  be self-adjoint operators. Then for all $h \in \mathcal{D}(AB)\cap  \mathcal{D}(BA)$ with $\|h\|=1$, we have 
\begin{align}\label{HR}
 \frac{1}{2} \left(\Delta _h(A)^2+	\Delta _h(B)^2\right)\geq \frac{1}{4} \left(\Delta _h(A)+	\Delta _h(B)\right)^2 \geq  \Delta _h(A)	\Delta _h(B)   \geq  \frac{1}{2}|\langle [A,B]h, h \rangle |.
\end{align}
\end{theorem}
In 1930,  Schrodinger improved Inequality (\ref{HR}) \cite{SCHRODINGER}. 
\begin{theorem} \cite{SCHRODINGER}
	(\textbf{Heisenberg-Robertson-Schrodinger  Uncertainty Principle}) \label{HRS}
	Let  $A:	\mathcal{D}(A)\subseteq \mathcal{H}\to  \mathcal{H}$ and $B:	\mathcal{D}(B)\subseteq \mathcal{H}\to  \mathcal{H}$  be self-adjoint operators. Then for all $h \in \mathcal{D}(AB)\cap  \mathcal{D}(BA)$ with $\|h\|=1$, we have 
	\begin{align*}
		\Delta _h(A)	\Delta _h(B)    \geq |\langle Ah, Bh \rangle-\langle Ah, h \rangle \langle Bh, h \rangle|
		=\frac{\sqrt{|\langle [A,B]h, h \rangle |^2+|\langle \{A,B\}h, h \rangle -2\langle Ah, h \rangle\langle Bh, h \rangle|^2}}{2}.
	\end{align*}	
\end{theorem}
Theorem \ref{HRS} leads to  the following question.
\begin{question}\label{Q}
	What is the  version of Theorem \ref{HRS} for three operators?
\end{question}

In this note, we answer Question \ref{Q} by deriving  an uncertainty principle for three operators acting on classes of Banach spaces, using trilinear forms. We note that there are uncertainty principles derived for three operators on Hilbert spaces, but ours differ from them \cite{SONGQIAO, KECHRIMPARISWEIGERT, QINFEILI, KECHRIMPARISWEIGERT2, DODONOV}.

\section{3-Heisenberg-Robertson-Schrodinger Uncertainty Principle}
The Heisenberg-Robertson-Schrodinger uncertainty principle requires the inner product to handle two operators; for three operators, we need a 3-product defined as follows. 
\begin{definition}
	Let $\mathcal{X}$ be a real Banach space with norm $\|\cdot\|$. A map $\langle \cdot, \cdot, \cdot \rangle: \mathcal{X}\times \mathcal{X}\times \mathcal{X}\to \mathbb{R}$ is said to be a \textbf{3-product} if following conditions hold.
	\begin{enumerate}[\upshape(i)]
		\item $\langle x, y, z \rangle =\langle \sigma(x), \sigma (y), \sigma (z) \rangle$ for all $x, y, z \in \mathcal{X}$,  for all bijections $\sigma:\{x, y, z\} \to \{x, y, z\}$. 
		\item $\langle \alpha x, y, z \rangle=\alpha \langle x, y, z \rangle$ for all $x, y, z \in \mathcal{X}$, for all $\alpha \in \mathbb{R}$.
		\item $\langle  x+w, y, z \rangle=\langle x, y, z \rangle+ \langle w, y, z \rangle$ for all $x, y, z, w \in \mathcal{X}$.
		\item $|\langle  x, y, z \rangle|\leq \|x\|\|y\|z\|$ for all $x, y, z \in \mathcal{X}$.
	\end{enumerate}
In this case, we say that $\mathcal{X}$ is a \textbf{3-product space}. 
\end{definition}
Following is the standard example we keep in mind. 
\begin{example}\label{SE}
	Let $(\Omega, \mu)$ be a measure space and $\mathcal{L}^3(\Omega, \mu)$ be the standard  real Lebesgue space. Generalized Holder's inequality says that
	\begin{align*}
		\int_{\Omega}|f_1(x)f_2(x)f_3(x)| \, d \mu (x)&\leq \left(\int_{\Omega}|f_1(x)|^3 \, d \mu (x)\right)^\frac{1}{3}\left(\int_{\Omega}|f_2(x)|^3 \, d \mu (x)\right)^\frac{1}{3}\left(\int_{\Omega}|f_3(x)|^3 \, d \mu (x)\right)^\frac{1}{3}\\
		&<\infty, \quad \forall f_1, f_2, f_3 \in \mathcal{L}^3(\Omega, \mu).
	\end{align*}
	Therefore  $\mathcal{L}^3(\Omega, \mu)$ is a 3-product space equipped with 3-product
	\begin{align*}
		\langle f_1, f_1, f_3 \rangle \coloneqq \int_{\Omega}f_1(x)f_2(x)f_3(x) \, d \mu (x), \quad \forall f_1, f_2, f_3 \in \mathcal{L}^3(\Omega, \mu).
	\end{align*}
\end{example}
We next introduce the notion of self-adjointness for operators on 3-product spaces.
\begin{definition}
Let  $\mathcal{X}$ be a 3-product space. 	A possibly unbounded  linear operator $A:\mathcal{D}(\mathcal{X})\subseteq \mathcal{X}\to \mathcal{X}$ is said to be \textbf{3-self-adjoint} if 
\begin{align*}
\langle  Ax, y, z \rangle=\langle  x, Ay, z \rangle	=\langle  x, y, Az \rangle, \quad \forall x, y, z \in \mathcal{D}(\mathcal{X}).
\end{align*}
\end{definition}
\begin{example}
	Consider $\mathbb{R}^n$ with the 3-product 
	\begin{align*}
		\left \langle (x_j)_{j=1}^n, (y_j)_{j=1}^n, (z_j)_{j=1}^n\right\rangle\coloneqq \sum_{j=1}^{n}x_jy_jz_j, \quad \forall (x_j)_{j=1}^n, (y_j)_{j=1}^n, (z_j)_{j=1}^n\in \mathbb{R}^n.
	\end{align*}
Let $a_1,\dots,  a_n$ be any real numbers. Define 
\begin{align*}
	A:\mathbb{R}^n \ni (x_j)_{j=1}^n \mapsto (a_jx_j)_{j=1}^n \in \mathbb{R}^n.
\end{align*}
Then $A$ is 3-self-adjoint. 
\end{example}
\begin{example}
Consider $\ell^3(\mathbb{N})$  (as a real sequence space) with the 3-product 
	\begin{align*}
	\left \langle \{x_n\}_{n=1}^\infty, \{y_n\}_{n=1}^\infty, \{z_n\}_{n=1}^\infty\right\rangle\coloneqq \sum_{n=1}^{\infty}x_ny_nz_n, \quad \forall \{x_n\}_{n=1}^\infty, \{y_n\}_{n=1}^\infty, \{z_n\}_{n=1}^\infty\in \ell^3(\mathbb{N}).
\end{align*}
Let $\{x_n\}_{n=1}^\infty$ be a bounded real sequence. Define 
\begin{align*}
	A: \ell^3(\mathbb{N}) \ni \{x_n\}_{n=1}^\infty
	 \mapsto \{a_nx_n\}_{n=1}^\infty \in \ell^3(\mathbb{N}).
\end{align*}
Then $A$ is 3-self-adjoint. 
\end{example}
\begin{example}
We continue from Example \ref{SE}. Let $\phi  \in \mathcal{L}^\infty (\Omega, \mu)$. Define 
\begin{align*}
	A:\mathcal{L}^3(\Omega, \mu)\ni f \mapsto Af\in \mathcal{L}^3(\Omega, \mu);\quad Af:\Omega\ni \alpha \mapsto (Af)(\alpha)\coloneqq \phi (\alpha)f(\alpha) \in \mathbb{R}.
\end{align*}
Then $A$ is 3-self-adjoint. 
\end{example}
Let $A:\mathcal{D}(\mathcal{X})\subseteq \mathcal{X}\to \mathcal{X}$ be a 3-self-adjoint operator. For $x \in \mathcal{D}(A)$ with $\langle x, x,  x \rangle=1$, define the \textbf{3-uncertainty} of $A$ at the point $x$ as 
\begin{align*}
	\Delta _x(3, A)\coloneqq \|Ax-\langle Ax, x, x \rangle x \|. 
\end{align*}
\begin{theorem}
(\textbf{3-Heisenberg-Robertson-Schrodinger  Uncertainty Principle})
Let  $\mathcal{X}$ be a 3-product space. Let  $A:	\mathcal{D}(A)\subseteq \mathcal{X}\to  \mathcal{X}$,  $B:	\mathcal{D}(B)\subseteq \mathcal{X}\to  \mathcal{X}$ and $C:	\mathcal{D}(C)\subseteq \mathcal{X}\to  \mathcal{X}$ be  possibly unbounded 3-self-adjoint operators. Then for all 
\begin{align*}
	x \in \mathcal{D}(ABC)\cap\mathcal{D}(ACB) \cap \mathcal{D}(BAC)\cap\mathcal{D}(BCA) \cap \mathcal{D}(CAB)\cap\mathcal{D}(CBA)
\end{align*}
with $\langle x, x, x \rangle =1$, we have 	
\begin{align*}
	&\frac{1}{27}\left(\Delta _x(3, A)	+\Delta _x(3, B)+\Delta _x(3, C)\right)^3\geq \Delta _x(3, A)	\Delta _x(3, B)	\Delta _x(3, C)\geq \\
	&|\langle (ABC-\langle Ax, x, x \rangle BC-\langle Bx, x, x \rangle AC-\langle Cx, x, x \rangle AB)x, x, x\rangle +2\langle Ax, x, x \rangle\langle Bx, x, x \rangle\langle Cx, x, x \rangle|.
\end{align*}
\end{theorem}
\begin{proof}
First inequality follows from AM-GM inequality for three positive reals. Given 
\begin{align*}
	x \in \mathcal{D}(ABC)\cap\mathcal{D}(ACB) \cap \mathcal{D}(BAC)\cap\mathcal{D}(BCA) \cap \mathcal{D}(CAB)\cap\mathcal{D}(CBA)
\end{align*}
with $\langle x, x, x \rangle =1$, set 
	\begin{align*}
		a\coloneqq \langle Ax, x, x \rangle, \quad b \coloneqq \langle Bx, x, x \rangle, \quad c \coloneqq \langle Cx, x, x \rangle.
	\end{align*}
Then 
\begin{align*}
&\Delta _x(3, A)\Delta _x(3, B)	\Delta _x(3, C)\geq |\langle Ax-ax, Bx-bx, Cx-cx\rangle| \\
&=|\langle ABCx, x ,x\rangle -\langle (aBC+bAC+cAB)x, x, x \rangle +\langle (abC+bcA+caB)x, x, x \rangle-abc|\\
&=|\langle ABCx, x ,x\rangle -\langle (aBC+bAC+cAB)x, x, x \rangle +2abc|\\
&=|\langle (ABC-\langle Ax, x, x \rangle BC-\langle Bx, x, x \rangle AC-\langle Cx, x, x \rangle AB)x, x, x\rangle +2\langle Ax, x, x \rangle\langle Bx, x, x \rangle\langle Cx, x, x \rangle|.
\end{align*}
\end{proof}

 \bibliographystyle{plain}
 \bibliography{reference.bib}

\begin{thebibliography}{10}

\bibitem{CASSIDY}
David~C. Cassidy.
\newblock Heisenberg, uncertainty and the quantum revolution.
\newblock {\em Sci. Amer.}, 266(5):106--112, 1992.

\bibitem{CONDON}
E.~U. Condon.
\newblock Remarks on uncertainty principles.
\newblock In {\em Selected Scientific Papers of E.U. Condon}, pages 93--96.
  Springer, New York, 1991.

\bibitem{DEBNATHMIKUSINSKI}
Lokenath Debnath and Piotr Mikusi\'{n}ski.
\newblock {\em Introduction to {H}ilbert spaces with applications}.
\newblock Academic Press, Inc., San Diego, CA, 1999.

\bibitem{DODONOV}
V.~V. Dodonov.
\newblock Variance uncertainty relations without covariances for three and four
  observables.
\newblock {\em Phys. Rev. A}, 97(2):022105, 2018.

\bibitem{HEISENBERG}
W.~Heisenberg.
\newblock The physical content of quantum kinematics and mechanics.
\newblock In {\em Quantum Theory and Measurement}, pages 62--84. Princeton
  University Press, Princeton, NJ, 1983.

\bibitem{KECHRIMPARISWEIGERT}
Spiros Kechrimparis and Stefan Weigert.
\newblock Heisenberg uncertainty relation for three canonical observables.
\newblock {\em Phys. Rev. A}, 90(6):062118, 2014.

\bibitem{KECHRIMPARISWEIGERT2}
Spiros Kechrimparis and Stefan Weigert.
\newblock Geometry of uncertainty relations for linear combinations of position
  and momentum.
\newblock {\em J. Phys. A}, 51(2):025303, 18, 2018.

\bibitem{OZAWA}
Masanao Ozawa.
\newblock Heisenberg’s original derivation of the uncertainty principle and
  its universally valid reformulations.
\newblock {\em Current Science}, 109(11):2006--2016, 2015.

\bibitem{QINFEILI}
Fei~SM. Qin, HH. and Li-Jost.
\newblock Multi-observable uncertainty relations in product form of variances.
\newblock {\em Scientific Reports}, (6):1--7, 2016.

\bibitem{ROBERTSON}
H.~P. Robertson.
\newblock The uncertainty principle.
\newblock {\em Phys. Rev.}, 34(1):163--164, 1929.

\bibitem{SCHRODINGER}
E.~Schr\"{o}dinger.
\newblock About {H}eisenberg uncertainty relation.
\newblock {\em Bulgar. J. Phys.}, 26(5-6):193--203, 1999.

\bibitem{SONGQIAO}
Qiu-Cheng Song and Cong-Feng Qiao.
\newblock Stronger {S}chr\"{o}dinger-like uncertainty relations.
\newblock {\em Phys. Lett. A}, 380(37):2925--2930, 2016.

\bibitem{VONNEUMANNBOOK}
John von Neumann.
\newblock {\em Mathematical foundations of quantum mechanics}.
\newblock Princeton University Press, Princeton, NJ, 2018.

\end{thebibliography}

\end{document}